\documentclass[12pt]{amsart}
\usepackage{amssymb}
\usepackage{amsmath, amscd}
\usepackage{amsthm}
\usepackage{xcolor}
\usepackage{ulem}

\newtheorem{theorem}{Theorem}[section]

\newtheorem{proposition}[theorem]{Proposition}
\newtheorem{lemma}[theorem]{Lemma}

\newtheorem{corollary}[theorem]{Corollary}
\theoremstyle{definition}

\newtheorem{example}[theorem]{Example}
\newtheorem{definition}[theorem]{Definition}

\def\val#1{\vert #1 \vert}

\def\bbv#1{{\bf #1}}
\def\Z{{\mathbb Z}}
\def\Q{{\mathbb Q}}
\def\N{{\mathbb N}}

\usepackage[utf8]{inputenc}
\usepackage{amssymb}
\usepackage{amsmath, amscd}
\usepackage{amsthm}
\usepackage{xcolor}

\begin{document}

\author[P.V. Danchev]{Peter V. Danchev}
\address{Institute of Mathematics and Informatics, Bulgarian Academy of Sciences \\ "Acad. G. Bonchev" str., bl. 8, 1113 Sofia, Bulgaria.}
\email{danchev@math.bas.bg; pvdanchev@yahoo.com}
\author[P.W. Keef]{Patrick W. Keef}
\address{Department of Mathematics, Whitman College, 345 Boyer Avenue, Walla Walla, WA, 99362, United States of America.}
\email{keef@whitman.edu}

\title[Bassian-Finite Abelian Groups]{Bassian-Finite Abelian Groups}

\maketitle

\begin{abstract} We introduce a new class of Abelian groups which lies strictly between the classes of co-Hopfian groups and Dedekind-finite groups, calling these groups {\it Bassian-finite}. We prove the surprising fact that in the torsion case the Bassian-finite property coincides with the co-Hopficity, thus extending a recent result by Chekhlov-Danchev-Keef in Siber. Math. J. (2026), and we construct a torsion-free Bassian-finite group which is {\it not} co-Hopfian as well as a Dedekind-finite group which is {\it not} Bassian-finite. Some other closely relevant things are also established. E.g., we extend a construction of a countable Butler group that is {\it not} completely decomposable, due to Arnold-Rangaswamy in Boll. Un. Mat. Ital. (2007), to find a Butler group of countably infinite rank which is Bassian-finite, but {\it not} completely decomposable.
\end{abstract}

\section{Introduction and Motivation}\label{one}

Herein, all groups considered will be Abelian and additively written. Unexplained notation and terminology will be standard and can be found, for instance, in \cite{F}. The letter $G$ will denote an arbitrary group whose torsion part will be denoted by $T$ and, if $p$ is a prime, the $p$-torsion component of $G$ will be designed by $T_p$. Of course, $\omega$ denotes the first infinite ordinal and cardinal.

\medskip

Mimicking \cite{BP}, a group $G$ is {\it Dedekind-finite}, provided there exists no isomorphism from $G$ to any proper direct summand of $G$. So, it is quite natural to examine what is the behavior of those group $G$ for which there exists no monomorphism from $G$ to any proper direct summand of $G$. It is immediately seen that every such a group is Dedekind-finite.

On the other hand, we recall the notion introduced in \cite{CDG} of a {\it Bassian group} as being such a group that cannot be embedded in a proper homomorphic image of itself. Thus, it is immediate that each Bassian group possesses the described above new property. Thereby, saving the same terminology, we call the newly defined groups as {\it Bassian-finite}. Therefore, we have the (strict) implications:

\medskip

\noindent{\centerline{Bassian $\Rightarrow$ Bassian-finite $\Rightarrow$ Dedekind-finite.}}

\medskip

So, we are left to consider the exact definition: {\it A group $G$ called Bassian-finite if there exists no monomorphism from $G$ to any proper direct summand of $G$}.

\medskip

In the other vein, imitating \cite{CDK}, we recall also that $G$ is said to be {\it relatively co-Hopfian} if, for all monomorphisms $\phi:G\to G$, $\phi(G)$ is an essential subgroup of $G$. In other words, for all monomorphisms $\phi:G\to G$ and non-zero subgroups $B\leq G$, we must have $\phi(G)\cap B\ne \{0\}$. This is a common extension of the standard notion of a co-Hopfian group: a group is {\it co-Hopfian}, provided it is a group that is not isomorphic to any of its proper subgroups.

\medskip

We now contrast that with the following useful observation.

\begin{proposition}\label{condition}
A group $G$ is Bassian-finite if, and only if, for all monomorphisms $\phi:G\to G$ and non-zero direct summands $B\leq G$, we have $\phi(G)\cap B\ne \{0\}$.
\end{proposition}

\begin{proof}
Suppose $G$ fails to be Bassian-finite. Then, there is a decomposition $G=A\oplus B$, where $B\ne \{0\}$ and $\phi:G\to A$ is a monomorphism. In particular, $\phi(G)\cap B\leq A\cap B=\{0\}$.

Conversely, if the given condition fails, we can find $\phi:G\to G$ which is a monomorphism and $B\ne \{0\}$ which is a direct summand of $G$ such that $\phi(G)\cap B=\{0\}$. Say $G=A\oplus B$ and let $\pi:G\to A$ be the usual projection onto the direct summand $A$. It easily follows from $\phi(G)\cap B=\{0\}$ that $\pi \circ \phi=: \phi':G\to A$ is also a monomorphism and $\phi'(G)\leq A$, showing that $G$ is not Bassian-finite, as needed.
\end{proof}

Since any non-zero direct summand is always a non-zero subgroup, we derive the following direct consequence.

\begin{corollary}\label{relativelycH}
If $G$ is relatively co-Hopfian, then it is Bassian-finite.
\end{corollary}

So, our (strict) inclusions listed above take the more precise form:

\medskip

\noindent{\centerline{Bassian $\Rightarrow$ (Relatively) co-Hopfian $\Rightarrow$ Bassian-finite $\Rightarrow$ Dedekind-finite.}}

\medskip

Our work is organized in what follows thus: In next section, we state and prove our principal results some of which sound like this -- in Proposition~\ref{torsion} we establish that the property Bassian-finite is equivalent to being co-Hopfian. In the subsequent Proposition~\ref{indec} we show that torsion-free co-Hopfian groups are always Bassian-finite. Next, in Proposition~\ref{cdconditions}, we prove that completely decomposable torsion-free groups which satisfy the ascending type condition are Bassian-finite, and vice versa, and in Example~\ref{dedfin} we construct a completely decomposable Dedekind-finite group that is {\it not} Bassian-finite. Furthermore, in Example~\ref{basfin}, we exhibit a torsion-free Bassian-finite group that is {\it not} co-Hopfian. Defining the notion of "finite injective rank", we succeed to build in our final Example~\ref{but} the existence of a countably infinite Butler group that is Bassian-finite, but {\it not} completely decomposable. We finish off our research with the closely related concept of being "co-Bassian-finite".

\section{Results and Examples}

We first study in detail the situation for torsion groups. The first statement is handled in the usual way, so we omit the details.

\begin{proposition}\label{torsionequivalent} Let $T=\bigoplus_{p\in \mathcal P} T_p$ be a torsion group.
Then, $T$ is Bassian-finite if, and only if, each direct component $T_p$ is Bassian-finite.
\end{proposition}

Our next rather curious assertion shows that, for $p$-torsion groups, the converse of Corollary~\ref{relativelycH} is valid, thus expanding the corresponding result from \cite{CDK}.

\begin{proposition}\label{torsion}
If $T=T_p$ is a $p$-group, then the following items are equivalent:

(a) $T$ is Bassian-finite.

(b) $T$ is co-Hopfian.

(c) $T$ is relatively co-Hopfian.
\end{proposition}

\begin{proof} That (b) implies (c) is automatic, and that (c) implies (a) is precisely Corollary~\ref{relativelycH}.

To show (a) yields (b), we argue by contrapositive. Supposing $T$ is not co-Hopfian, we wish to show $T$ also fails to be Bassian-finite. If $X\neq 0$, then it is readily checked that $X^{(\omega)}$ is not Bassian-finite (in fact, since $X^{(\omega)}\cong X\oplus X^{(\omega)}$, it is not even Dedekind-finite). Therefore, we may assume $T$ does not have a direct summand isomorphic to such a direct sum. It now follows that the Ulm function of $T$ satisfies the property $f_T(n)$ is finite whenever $n<\omega $ or $n=\infty$.

However, since $G$ is not co-Hopfian, there is an endomorphism $\phi:G\to G$ that is injective, but not surjective.
We claim that there is an element $y\in G[p]\setminus \phi(G)$ such that $\val y<\omega$. Before verifying this claim, we note how it forces the full result: There is, clearly, a cyclic direct summand $B=\langle w\rangle$ with $B[p]=\langle y\rangle$. Consequently, $B\cap \phi(G)=\{0\}$, so that thanks to Proposition~\ref{condition}, $B$ is not Bassian-finite, completing the whole proof.

For every $j<\omega$, let
$$
S_j=G[p]/(p^j G)[p].
$$
Evidently, for all $j<\omega$, $\phi$ induces an endomorphism
$$
                            \phi_j: S_j \to S_j.
$$

Suppose that, for some $j<\omega $, that $\phi_j$ is not surjective. If $$\overline{0}\ne y+(p^j G)[p]\not\in \phi_j(S_j),$$ then it follows that $\val y<j<\omega$ and hence $$y\not\in \phi(G[p])=\phi(G)\cap G[p]$$ establishes our claim.

On the other hand, suppose that, for all $j<\omega$, that $\phi_j$ is surjective. Since, by what we have shown above, $f_G(n)$ is finite for all $n<\omega$, it follows that each $S_j$ is finite. It also follows that each $\phi_j$ is injective too. In other words, for $x\in G[p]$ and $j<\omega$, that $\val x<j$ implies $\val {\phi(x)}<j$. This readily assures that $\phi(x)$ preserves all {\it finite} $p$-heights of elements of $G[p]$. This gives that $\phi(G)$ will be a {\it pure} subgroup $G$. If $G[p]=\phi(G[p])$, then the purity of $\phi(G)$ in $G$ would imply that $\phi(G)=G$, which would contradict that $\phi$ is {\it not} surjective.

Therefore, there is an $x\in G[p]\setminus \phi(G[p])$; we want to verify that the existing element $y\in G[p]\setminus \phi(G[p])$ can be found of finite $p$-height, i.e., $\val y=n<\omega$: Certainly, if $\val x$ is already finite, we may just let $y=x$. Otherwise, $\val x\geq \omega$. If $n<\omega$ with $f_G(n)\ne 0$, then let $z'\in (p^n G)[p]/(p^{n+1} G)[p]$. It thus follows that $$\phi(z')=: z\in \phi((p^n G)[p])\setminus (p^{n+1} G)[p].$$ So, $\val z=n$, and if we let $y=x+z$, then $y\in G[p]\setminus \phi(G)$ with $\val y=n<\omega$, as required.
\end{proof}

That is why, with a combination of Propositions~\ref{torsionequivalent} and \ref{torsion} at hand, one can completely describe Bassian-finite groups in terms of co-Hopfian groups.

\medskip

We now turn to non-torsion groups in which case the structural is rather more complicated than usual. We first begin with the following simple but helpful fact.

\begin{proposition} Suppose $G$ is a group with a maximal torsion subgroup $T$.

(a) If $T$ and $G/T$ are both Bassian-finite, then $G$ is Bassian-finite too.

(b) If $G$ splits, that is, $G\cong T\oplus (G/T)$, then $G$ is Bassian-finite if, and only if, both $T$ and $G/T$ are Bassian-finite.
\end{proposition}

\begin{proof}
For (a), let $G=A\oplus B$ and suppose $\phi:G\to A$ is a monomorphism; we thus need to show $B=\{0\}$. We have a decomposition
$$
                     T=T_A\oplus T_B.
$$
Note that $\phi$ restricts to a monomorphism $\phi_T:T\to T_A$. Therefore, since $T$ is Bassian-finite, we must have $T_B=\{0\}$, which means that $T=T_A\leq A$.

Now, it is plainly checked that the injection $\phi:G\to A$ induces another injection
$\overline \phi:G/T\to A/T_A$. Since
$$
            G/T\cong (A/T_A)\oplus (B/T_B)= (A/T_A)\oplus B
$$
is Bassian-finite, we can conclude that $B=\{0\}$, as desired.

For (b), sufficiency is an immediate consequence of (a), and necessity follows from the fact that a direct summand of a Bassian-finite group retains that property, as wanted.
\end{proof}

We have described in Proposition~\ref{torsion} the torsion Bassian-finite groups (or at least we have reduced them to the case of co-Hopfian groups). In the torsion-free case, we obtain some results, but they are unfortunately still far from the asked general ones.

\begin{proposition}\label{indec} Suppose $G$ is a torsion-free group. In any of the following cases, $G$ will be Bassian-finite:

(a) $G$ has finite rank.

(b) $G$ is indecomposable.

(c) $G$ is co-Hopfian.
\end{proposition}

\begin{proof}
Regarding (a), suppose $G=A\oplus B$ and $\phi:G\to A$ is a monomorphism; we need to show $B=\{0\}$. We deduce
$$
         r(G)=r(\phi(G))\leq r(A)\leq r(G).
$$
Consequently, $r(A)=r(G)$, so that $r(B)=\{0\}$, i.e., $B=\{0\}$, as claimed.

Points (b) and (c) are straightforward, so we drop off the necessary arguments.
\end{proof}

We have seen above that if a group is relatively co-Hopfian, then it is Bassian-finite, and in the case of torsion groups, the converse holds as well. In addition, any torsion-free group of finite rank will be both relatively co-Hopfian and Bassian-finite. Likewise, a free group is Bassian-finite precisely when it has finite torsion-free rank.

\medskip

The following assertion illustrates that this converse also holds for torsion-free groups that are completely decomposable (i.e., when they are a direct sum of groups of rank 1).

\begin{proposition}\label{cdconditions}
Suppose $G=\bigoplus_{i\in I}A_i$ is a completely decomposable torsion-free group, where each $A_i$ has rank 1 and type $\tau_i$. Then, the following conditions are equivalent:

(a) $G$ is Bassian-finite.

(b) $G$ satisfies the ascending type condition.

(c) $G$ is relatively co-Hopfian.
\end{proposition}

\begin{proof} The equivalence of (b) and (c) is exactly \cite[Proposition~2.35]{CDK}. To show they are equivalent to (a), notice that we have already observed that (c) leads to (a). Conversely, suppose that (b) fails. Then, there is a direct summand, say $A\leq G$, such that $A=\oplus_{i\in \N} A_i$, where $\tau_i\geq \tau_{i+1}$ for each $i\in \N$. If we can prove that $A$ fails to be Bassian-finite, then the same holds for $G$, establishing this implication. To that end, for each $i\in \N$ there is apparently a monomorphism $A_i\to A_{i+1}$. These can all be combined into a single monomorphism $\phi:A\to A$ such that $A_1\cap \phi(A)=0$. From this, it follows at once that $A$, and hence $G$, fails to be Bassian-finite, as asserted.
\end{proof}

With the aid of the last statement, we next exhibit the following construction.

\begin{example}\label{dedfin} There is a completely decomposable torsion-free group that is Dedekind-finite, but {\it not} Bassian-finite.
\end{example}

\begin{proof} Suppose, for each $i<\omega$, we have a torsion-free group $A_i$ of rank-one and type $\tau_i$, where
$
               \tau_0<\tau_1<\tau_2<\cdots.
$
Letting $G=\bigoplus_{i<\omega} A_i$, Proposition~\ref{cdconditions} tells us that $G$ is not Bassian-finite. 

On the other hand, if $G\cong G\oplus H$, where $H\ne 0$, then one sees that $H\cong  \bigoplus_{i\in J} A_i$, where $\emptyset\ne J\subseteq \omega$. However, if $i\in J$, then $G$ has one direct summand isomorphic to $A_i$, whereas $G\oplus H$ has (at least) two. This contradiction unambiguously shows that $G$ is Dedekind-finite, as pursued.
\end{proof}

The following strategically construction demonstrates that the converse of Corollary~\ref{relativelycH}, though valid for torsion-free groups of finite rank and torsion-free groups that are completely decomposable, does {\it not} extend to arbitrary torsion-free groups of infinite rank. Concretely, the following holds:

\begin{example}\label{basfin}
There is a torsion-free group $G$ that is Bassian-finite, but {\it not} relatively co-Hopfian.
\end{example}

\begin{proof}
For some prime $p$, let $\hat \Z_p$ denote the (additive group of the) ring consisting of all $p$-adic integers. Let $\alpha\in \hat\Z_p$ denote some element that is transcendental over $\Q\leq \Q\hat \Z_p$ (denoting the field of quotients of $\hat \Z_p$). Set
$$
            A:=\Z[\alpha]=\{b_0+b_1\alpha +b_2\alpha^2 +\cdots + b_k \alpha^k, k\in \N, b_i\in \Z\}.
$$
Since $\alpha$ is transcendental, one inspects that $A$ will be isomorphic to the ring of polynomials $\Z[x]$ in the indeterminate $x$.
Finally, put
$$
            G:=\{\beta\in \hat\Z_p: p^j \beta \in A\ \ {\rm for\ some\ }j\in \N\}
$$
be the $p$-pure closure of $A$.

If $\phi$ is any endomorphism of $G$, then since $G$ is checked to be dense and Hausdorff in the $p$-adic topology of $\hat \Z_p$, $\phi$ must uniquely extend to a homomorphism $\hat \Z_p\to \hat \Z_p$. And since $\gamma:=\phi(1)\in G$, we must have $\phi(\beta)=\beta\phi(1)=\beta\cdot \gamma$ for all $\beta\in G$. Therefore, the endomorphism ring of $G$ is a subring of $\hat \Z_p$ (indeed, it equals $G$, that is a fact we do not need in the sequel).

Consider now the endomorphism of $G$ defined by $\alpha(z)=\alpha \cdot z$ for all $z\in G$. Since $\hat \Z_p$ is an integral domain, this multiplication by $\alpha$ is known to be injective. And since $\langle 1\rangle \cap \alpha G=\{0\}$, the map $\alpha G$ is not essential in $G$, so that $G$ fails to be relatively co-Hopfian in accordance with the definition alluded to above.

However, since again $\hat \Z_p$ is an integral domain, it has no idempotents. Therefore, neither does the endomorphism ring of $G$; so, we conclude that $G$ is indecomposable. Furthermore, in virtue of Proposition~\ref{indec}, any indecomposable group is Bassian-finite, and thus $G$ is our pursued example.
\end{proof}

Nevertheless, the constructed above group seems to be neither completely decomposable nor its proper homomorphic image, we now explore in-depth what happens in the situation of so-called Butler groups as defined in \cite{Fu} under the presumption that it is of finite rank: a group $G$ said to be {\it Butler} if it is a pure subgroup of a completely decomposable group or, equivalently, it is an epimorphic image of a completely decomposable group. Specifically, we proceed with the following possible expansions for arbitrary infinitely countable ranks.

\medskip

Our discussion of some of these properties have been straightforward in the case of torsion-free groups of finite rank. If we examine those arguments, they appear to be based upon the following strengthening of the concept ``weakly co-finitely Hopfian", as defined in \cite{DK}, that is clearly satisfied by such groups:

\begin{definition}
The group $G$ is said to have {\it finite injective rank} if, for every endomorphism $\phi:G\to G$, $\phi$ is injective if, and only if, $G/\phi(G)$ is finite.
\end{definition}

Obviously, if $G$ has finite injective rank, then it is weakly co-finitely Hopfian and, in particular, it must be torsion-free and Hopfian. As noted above, if $G$ is torsion-free of finite rank, then it has finite injective rank. On the other hand, the $p$-adic integers, $\hat \Z_p$, will have finite injective rank, though its torsion-free rank is the continuum. The following claim compares this definition with others involving monomorphisms.

\begin{proposition} Suppose $G$ is a group with finite injective rank. Then, the following hold:

(a)  $G$ is relatively co-Hopfian.

(b)  $G$ is Bassian-finite.

(c)  $G$ is co-finitely Hopfian if, and only if, it is divisible of finite rank.
\end{proposition}

\begin{proof} (a) If $\phi:G\to G$ is a monomorphism, then since $G$ is torsion-free and $G/\phi(G)$ is torsion, $\phi(G)$ is essential in $G$.

(b) follows from (a) and Corollary~\ref{relativelycH}.

(c) Sufficiency being clear, suppose $G$ is not divisible. Then, for some prime $p$, $\phi(x)=px$ will be injective, but not surjective. Thus, $G/\phi(G)$ will be finite, but non-zero, so that $G$ will not be co-finitely Hopfian.

Next, suppose $G$ is divisible, but does not have finite rank. It follows from the next result below that $G$ cannot have finite injective rank, as requested.
\end{proof}

Infinite direct sums can never have finite injective rank as the following assertion shows.

\begin{proposition}
Suppose $G=\bigoplus_{i\in I} A_i$, where each $A_i\ne 0$. If $G$ has finite injective rank, then $I$ is finite.
\end{proposition}

\begin{proof}
Suppose the contrary that $I$ is infinite. If $A_i$ is not reduced for an infinite collection of $i\in I$, then $G$ has a direct summand isomorphic to $\Q^{(\omega)}$. Since this summand is not Hopfian, neither is $G$, which shows that $G$ does not have finite injective rank.

Otherwise, $G$ will have a direct summand of the form $G':=\bigoplus_{i<\omega}A_i$, where, for each $i<\omega$, there is a prime $p_i$ such that $p_i A_i\ne A_i$. If $\phi': G'\to G'$ is multiplication by $p_i$ on each $A_i$, then $\phi$ will be injective, but $G'/\phi'(G')$ will be infinite. Setting $\phi$ equaling $\phi'$ on $G'$ and the identity on a complementary direct summand, shows that $G$ does not have finite injective rank, as required.
\end{proof}

As an automatic consequence, we derive:

\begin{corollary} Suppose $G$ is a completely decomposable torsion-free group. Then, $G$ has finite injective rank if, and only if, it has finite torsion-free rank.
\end{corollary}

Furthermore, as mentioned before, $\hat \Z_p$ has finite injective rank, but not finite torsion-free rank. Are there broader classes of groups for which, like the last corollary, this must hold? In particular, if $G$ is a Butler group with finite injective rank, does $G$ have finite torsion-free rank? We will show that this surprisingly fails; i.e., there is a Butler group $R$ with finite injective rank, but countably infinite torsion-free rank.

\medskip

The following group appeared in \cite{AR} and we review its construction, though we will make a number of additional observations about it. Let $W$ be the collection of all ``words" in the alphabet $\{0,1\}$ (including $\emptyset$).  We will think of such a word representing, in base 2, an $n<\omega= \{0,1,2, \dots\}$ (agreeing that $\emptyset, 0, 00, \dots$ all represent 0).

If $w\in W$, let $\lambda(w)$ denote its length. So, for example, $0110$ and $110$ both represent the number 6, though they are different words, since they have different lengths. Set
$$
       L_w:=\{n<\omega: n\equiv w \pmod {2^{\lambda(w)}};
$$
So, $L_\emptyset = \N$ and $L_{1011}=\{11, 27, 43, 59, \dots\}$. If $n<\omega$, then $\{L_{0w}, L_{1w}\}$ is a partition of $L_{w}$, and
$$
             \{ L_w:w\in W, \lambda(w)=n\}
$$
is a partition of $\omega$.

For $w\in W$, define $\Q_w\leq \Q$ as follows: Let $\mathcal P=\{p_0, p_1, p_2, \dots\}$ be the collection of all primes. Put
$$P_w:=\{p_i\in {\mathcal P}: i\in L_w\}\subseteq \mathcal P,$$
and let $\Q_w$ be the integers localized at $P_w$. So, $q\in \Q_w$ uniquely when we can write $q=\frac{a}{b}$, where $(p, b)=1$ for all $p\in P_w$; in fact, it is well known that $\Q_w$ is a PID with primes $p\in P_w$. Note that $\Q_\emptyset=\Z$, $\Q_{0w}\cap\Q_{1w}=\Q_w$ and $\Q_{0w}+\Q_{1w}=\Q$ (see, for more details, \cite[Proposition~1]{AR}).

For each $n<\omega$, introduce a ring $R_n$ as follows: Let $w_1$, $w_2$, \dots, $w_{2^n}$ be the collection of all words of length $n$. Define
$$
R_n:=\bigoplus_{\lambda (w)=n} \Q_w\bbv e_w=\Q_{w_1}\bbv e_{w_1} \oplus \Q_{w_2}\bbv e_{w_2}\oplus \cdots \oplus \Q_{w_{2^n}}\bbv e_{w_{2^n}},
$$
where the $\bbv e_w$s are (orthogonal) idempotents. So, $\bbv e_{w_1}+\cdots + \bbv e_{w_{2^n}}$ is the multiplicative identity of $R_n$. Notice that $R_0=\Z \bbv e_{\emptyset}\cong \Z$, and we identify $\bbv e_\emptyset$ with the number 1.

If $n<\omega$, then, for all $w\in W$ of length $n$, if we map $\bbv e_{w}\mapsto \bbv e_{0w}+\bbv e_{1w}$, then this extends to an injective (unitary) ring homomorphism $R_n\to R_{n+1}$. If we interpret these all as inclusions, then since $\Q_w=\Q_{0w}\cap \Q_{1w}$, we can conclude that, as torsion-free Abelian groups, $R_n$ is a pure subgroup of $R_{n+1}$. The union is the direct limit of the $R_n$, that is, $R:= \cup_{n<\omega} R_n$ will be a ring. Since, as Abelian groups, each $R_n$ will be a reduced pure subgroup of $R$, it follows that $R$ is also reduced. Manifestly, for each $n$, we have
$$
            1 = \sum_{\lambda(w)=n} \bbv e_w.
$$
Further, for a fixed $n<\omega$, there will be an $R$-module decomposition:
$$
          R=\bigoplus_{\lambda (w)=n} R\bbv e_w=R \bbv e_{w_1}\oplus R \bbv e_{w_2}\oplus \cdots \oplus R \bbv e_{w_{2^n}}.
$$
It is also readily checked that, for $w, v\in W$ with $\lambda (w)=n\leq \lambda (v)$, we then can calculate
$$
            \bbv e_w\cdot \bbv e_v=\begin{cases} \bbv e_v, & {\rm if}\ v\restriction_n=w;\cr
                                                    0,     & {\rm otherwise},\cr
                                   \end{cases}
$$
and, moreover, that each
$$R\bbv e_w=\sum_{\stackrel{\lambda(v)\geq n}{v\restriction_n=w}}\Q_v\bbv e_v$$
will be a module over $\Q_w$.

\medskip

Unlike the approach in \cite {AR}, our discussion of $R$ will be centered on the following interesting property:

\begin{lemma} The ring $R$ is an E-ring, i.e., for every {\bf group} endomorphism $\phi:R\to R$, there is an $\alpha \in R$ such that
$$
                     \phi(x)=\alpha x, \ \ \ {\rm for\ all\ }x\in R
$$
(see the exact definition in \cite [Section~18.6]{F}).
\end{lemma}

\begin{proof} Looking at \cite[Theorem~18.6.3(ii)]{F}, it suffices to show that, if $\phi$ is a group endomorphism of $R$ such that $\phi(1)=0$, then $\phi=0$; so assume $\phi$ is such an endomorphism.

Observes that $\phi$ will induce a homomorphism $\overline \phi: R/\Z\to R$, and it will suffice just to prove that $\overline \phi=0$. And since $R$ itself is reduced, it will be enough to show that $R/\Z$ is (torsion-free) divisible. Indeed, since
$$
                     R/\Z= \bigcup_{n<\omega} R_n/\Z,
$$
we need to prove that each $R_n/\Z$ is divisible. We establish this by induction on $n$, using the (exact) sequence:
$$
          0 \to R_n/\Z\to R_{n+1}/\Z\to R_{n+1}/R_n \to 0.
$$
That is, if we can show that each $R_{n+1}/R_n$ is divisible, then our inductive step from the $n$th case to the $n+1$st case holds, thus giving the whole result.

\medskip

To that end, we subsequently compute that
\begin{align*}
       R_{n+1}/R_n&=\left(\bigoplus_{ {\lambda(w)=n+1}} \Q_w\bbv e_w\right)/ \left(\bigoplus_{\lambda(w)=n} \Q_w\bbv e_w\right)\cr
                  &  \cong \bigoplus_{ {\lambda(w)=n}} \left[ (\Q_{0w}\bbv e_{{0w}}\oplus\Q_{1w}\bbv e_{{1w}})/ \Q_w(\bbv e_w)\right]\cr
                  &  =\bigoplus_{ {\lambda(w)=n}}  \left[(\Q_{0w}\bbv e_{{0w}}\oplus\Q_{1w}\bbv e_{{1w}})/ \Q_w(\bbv e_{0w}+\bbv e_{1w})\right]\cr
                  &  \cong \bigoplus_{{\lambda(w)=n}} (\Q_{0w}+\Q_{1w})= \bigoplus_{{\lambda(w)=n}} \Q= \Q^{(2^n)}\cr
\end{align*}
is divisible, as wanted.
\end{proof}

The same proof demonstrates that, for all words $w\in W$, the quotient $R \bbv e_w/\Q_w \bbv e_w$ is divisible; the above argument was just for $w=\emptyset$. However, it is easy to see that $R$ is commutative, since it is the union of the commutative rings $R_n$.

This brings us to our main reason for being interested in this construction as recorded in the following.

\begin{proposition} As a group, $R$ has finite injective rank.
\end{proposition}

\begin{proof}
If $\alpha\in R$, let $\phi_\alpha(x)=\alpha x$ for all $x\in R$. It follows from $R=\cup_{n<\omega} R_n$ that, if $\phi_\alpha$ is an arbitrary group endomorphism of $R$, then there is an $n\in \N$ such that $\alpha\in R_n$, so that there are $\gamma_w\in \Q_w$ ($\lambda(w)=n$) such that
$$
                     \alpha = \sum_{\lambda(w)=n} \gamma_w \bbv e_w,   \eqno{(\dag)}
$$
where $\gamma_w=\alpha\cdot \bbv e_w$.
In this representation, we can read off several properties that $\phi_\alpha$ may or may not have:

\medskip

(a) $\phi_\alpha$ is injective if and only if $\alpha$ is not a zero divisor, if and only if each $\gamma_w\ne 0$.

\medskip

(b) $\phi_\alpha$ is bijective if and only if $\alpha$ is a unit if and only, if each $\gamma_w$ is a unit of $\Q_w$, if and only if multiplication by $\alpha$ is surjective; in particular, the group {\bf $R$ is Hopfian} in the usual sense.

\medskip

(c) $\phi_\alpha$ is an idempotent if and only if each $\gamma_w$ is either 0 or 1 in $\Q_w$.

\medskip

Assume again that $\alpha$ is as in $(\dag)$. For each $w\in W$ in that sum, there is an exact sequence
$$0 \to \Q_w \bbv e_w\to R\bbv e_w \to (R\bbv e_w)/(\Q_w \bbv e_w)\to 0,$$
where, again, the right-hand group is isomorphic to $\Q^{(\kappa)}$ for some cardinal $\kappa$. It follows from this that
$$
              R\bbv e_w /\gamma_w R\bbv e_w \cong \begin{cases}  \Q_w \bbv e_w/\gamma_w \Q_w \bbv e_w \cong \Q_w/ \gamma_w \Q_w, & {\rm if }\  \gamma_w\ne 0\cr R\bbv e_w, & {\rm if }\ \gamma_w= 0\cr\end{cases}
$$
It, therefore, follows also that
$$
        R/\alpha R \cong \bigoplus_{\lambda(w)=n} R\bbv e_w /\gamma_w R\bbv e_w
$$
is finite if and only if each $\gamma_w\ne 0$, if and only if $\alpha$ is not a zero divisor, if and only if $\phi_\alpha$ is injective, as asked for.
\end{proof}

In \cite{AR} it was observed that, as an additive Abelian group, this group $R$ {\it is a $B_2$-group, and hence a $B_1$-group} (this followed since $R$ is the union of the pure subgroups $R_n$, each of which is clearly Butler). On the other hand, a pretty simple argument showed that this $R$ {\it cannot be embedded as a pure subgroup of a completely decomposable group} (see \cite [Example~1]{AR}).

\medskip

In addition, in \cite[Theorem~14.9.3]{F} it was noted that such a construction results in a group that is {\it super-decomposable} in the sense that $R$ does {\bf not} have a non-zero indecomposable direct summand -- in fact, the above characterization of idempotents in $R$, together with the relation $\bbv e_w=\bbv e_{0w}+\bbv e_{1w}$, easily implies that any non-zero idempotent is the sum of two other, non-zero idempotents.

\medskip

So, for example, comparing with Proposition~\ref{indec}, $R$ is a Bassian-finite group even though it satisfies none of the three conditions listed there.

\medskip

Next, observe that the group $R$ {\it does not} {\bf satisfy the ascending type condition}; because, for instance,
$$\tau(\bbv e_\emptyset)<\tau(\bbv e_0)<\tau(\bbv e_{00})<\tau(\bbv e_{000})<\tau(\bbv e_{0000})<\cdots ,$$ that is,
$$\tau(\Z)<\tau(\Q_0)<\tau(\Q_{00})<\tau(\Q_{000})<\tau(\Q_{0000})<\cdots.$$

Now, comparing with Proposition~\ref{cdconditions} stated and proved above, the so-constructed group $R$ satisfies (c), and hence (a), but however does {\it not} satisfy (b). So, a countable Butler group, may be relatively co-Hopfian without having the ascending type property. And certainly a homogeneous completely decomposable group of countably infinite rank will have the ascending type property without being Bassian-finite, so that it is also {\it not} relatively co-Hopfian.

\medskip

Thereby, we arrive at the following statement.

\begin{example}\label{but} There is a Butler group of infinitely countable rank which is Bassian-finite, but {\it not} completely decomposable.
\end{example}

\section{Related Concepts}

Similarly to the notion of a ``{\it Bassian-finite}" group, it is quite logical to consider the reciprocal one of a ``{\it co-Bassian-finite}" group $G$, which means that there exists no epimorphism from $G$ to any proper direct summand of $G$.

\medskip

The usual categorical arguments again shows that a direct summand of a group with this property retains that property. However, the above definition of {\it co-Bassian-finite} is sort of trivial as it stands:

\begin{proposition}
A group is co-Bassian-finite if, and only if, it is indecomposable.
\end{proposition}

\begin{proof}
If $G$ is an indecomposable group, then it has no proper direct summands, and so is vacuously co-Bassian-finite.

Conversely, if the co-Bassian-finite group $G$ is assumed in a way of contradiction to be not indecomposable, then there is a decomposition $G=A\oplus B$, where $A\ne \{0\}\ne B$. Then, the natural projection $G\to A$ is an epimorphism, so that $G$ is really not co-Bassian-finite, as expected.
\end{proof}


\medskip

\begin{thebibliography}{99}

\bibitem{AR}
D.M. Arnold and K.M. Rangaswamy, {\it A note on countable Butler groups}, Boll. Unione Mat. Ital. (Sez. B - Artic. Ric. Mat.) Ser. 8 {\bf 10}(3) (2007), 605--611.

\bibitem{BP}
R.A. Beaumont and R.S. Pierce, {\it Isomorphic direct summands of Abelian groups}, Math. Ann. {\bf 153} (1964), 21--37.

\bibitem{CDG}
A.R. Cheklov, P.V. Danchev and B. Goldsmith, {\it On the Bassian property for Abelian groups}, Arch. Math. Basel {\bf 117}(6) (2021), 593--600.

\bibitem{CDK}
A.R. Chekhlov, P.V. Danchev and P.W. Keef, {\it Two generalizations of co-Hopfian Abelian groups}, Sib. Math. J. {\bf 67}(1) (2026).

\bibitem{DK}
P.V. Danchev and P.W. Keef, {\it On some versions of Hopficity for Abelian groups}, submitted.

\bibitem{Fu}
L. Fuchs, {\it Butler groups of infinite rank}, J. Pure Appl. Algebra {\bf 98} (1995), 25--44.

\bibitem{F}
L. Fuchs, Abelian Groups, Springer, Switzerland (2015).

\end{thebibliography}
\end{document}